\documentclass{article}
\usepackage{graphicx} 
\usepackage{amsfonts,epsf,amsmath,amssymb,amsthm}
\usepackage{epsfig}
\usepackage{latexsym}
\usepackage{tikz}
\usepackage{tkz-graph}
\usetikzlibrary{shapes.geometric}
\tikzstyle{vertex}=[circle,fill, draw, inner sep=0pt, minimum size=6pt]

\newtheorem{theorem}{\bf Theorem}[section]
\newtheorem{corollary}[theorem]{\bf Corollary}
\newtheorem{lemma}[theorem]{\bf Lemma}

\newtheorem{proposition}[theorem]{\bf Proposition}

\begin{document}

\title{On finite nilpotent groups with the same enhanced power graph}

\let\cleardoublepage\clearpage
\author{M. Mirzargar$^{a,b}$, S. Sorgun$^{a,*}$, M. J. Nadjafi-Arani$^{a,b}$ \\
	\small $^{a}$Department of Mathematics, Nevşehir Hacı Bektaş Veli University, 50300, Nevşehir, Turkey \\
	\small $^{b}$
Faculty of Science, Mahallat Institute of Higher Education, Mahallat, Iran \\
\small $^{*}$Corresponding author \tt{srgnrzs@gmail.com; ssorgun@nevsehir.edu.tr } \\
	\small \tt{m.mirzargar@mahallat.ac.ir} \\
    \small \tt{mjnajafiarani@mahallat.ac.ir } \\
}	
	\maketitle
	
	\begin{abstract}
		\noindent
        The enhanced power graph of a group $G$ is the graph $P_e(G)$ whose vertex set is $G$, such that two distinct vertices $x$ and $y$, are adjacent if $\langle x, y\rangle$ is cyclic.  
		In this paper, we analyze the structure of the enhanced power graph of a finite nilpotent group in terms of the enhanced power graphs of its Sylow subgroups. We establish that for two nilpotent groups, their enhanced power graphs are isomorphic if and only if the enhanced power graphs of their Sylow subgroups are isomorphic. 
        Additionally, we identify specific nilpotent groups for which the enhanced power graphs uniquely characterize the group structure, meaning that if $P_e(G)\cong P_e(H)$ then $G \cong H$. Finally, we extend these results to power graphs and cyclic graphs.
		\bigskip
		\noindent
		
        \\
	\end{abstract}

\maketitle

\section{Introduction}
The structural interplay between algebraic groups and combinatorial graphs constitutes a foundational nexus in mathematical research. Generating graphs from
 semigroups and groups has a long history. One of the oldest studies of graphs defined on the elements of a group is the Cayley graph \cite{budden1985cayley}.  Various graphs can be defined for a group based on its structural properties, including the commuting graph,  power graph, enhanced power graph, and non-cyclic graph.
 The study of graphs associated with groups and other algebraic structures is highly significant due to their numerous valuable applications, particularly in automata theory, as discussed in books \cite{kelarev2002ring, kelarev2003graph, kelarev2009cayley}. Most of these graphs have natural labelings, making this area a part of the broader research field of graph labelings. For more details, see publications on graph labelings \cite{ngurah2021distance, indriati2020totally}. 

 The enhanced power graph of any group is the simple graph whose vertex set consists of all the elements of the group. Two elements $x$ and $y$ are adjacent if they belong to the same cyclic subgroup.
 The term enhanced power graph of a group was introduced in \cite{aalipour2016structure}. The enhanced power graph is a subgraph that exists as an intermediate structure between the commuting graph 
 and the undirected power graph of the group. Kelarev and Quinn in \cite{kelarev2000combinatorial} introduced the directed power graph of a group as a digraph where vertices represent group elements, and there is a directed edge from $y$ to $x$ if $x$ is a power of $y$. 
 The undirected power graph for  semigroups was
 introduced by Chakrabarty in \cite{chakrabarty2009undirected}. Groups represent a special case of semigroups where these concepts are specifically defined. For a detailed historical overview and the relationships between these concepts, see \cite{cameron2011power, cameron2010power, abawajy2013power, mirzargar2019survey}.
 
 The commuting graph is the graph whose vertex set is all elements of a group, and two distinct elements $x, y$ are adjacent if $xy = yx$. The commuting graph is the complement of the non-commuting graph, a concept initially explored by Erdős \cite{neumann1976problem}. Similarly, the enhanced power graph is closely related to the complement of the non-cyclic graph (which we refer to as the cyclic graph), with a difference in the vertex set.
 The non-cyclic graph of a group, introduced in \cite{abdollahi2007noncyclic,abdollahi2009non}, as the graph with two vertices $x$ and $y$ are adjacent if the subgroup $\langle x, y\rangle$ is non-cyclic. In this definition, isolated vertices are excluded.


  In some papers, the enhanced power graph is also called the cyclic graph \cite{ma2013cyclic}.
  It is also called the deleted enhanced power graph when the identity element is removed \cite{costanzo2021cyclic} or the proper enhanced power graph \cite{bera2022proper} when vertices joined to all other vertices are deleted. The variation in terminology and the close relationship between these graphs have resulted in similar findings appearing in the literature. For instance, Theorem 19 in \cite{ma2013cyclic}, and Corollary 38 in \cite{aalipour2016structure}, present similar results. To avoid this issue, we explicitly mention any results in this paper that are similar to those in other works.

\subsection{ Terminology and notation}

A graph $\Gamma=(V(\Gamma), E(\Gamma))$ is a structure with the vertex set $V(\Gamma)$ and edge set $E(\Gamma)$. If \( \{x, y\} \in E(\Gamma) \), we denote this by \( x \sim y \). The graph
 $\Gamma_1=(V_1, E_1)$ is a \textbf{subgraph} of $\Gamma=(V(\Gamma),E(\Gamma))$, if $V_1\subseteq V(\Gamma)$ and $E_1\subseteq E(\Gamma)$.
If $U\subseteq V(\Gamma)$, then the \textbf{induced subgraph} of $\Gamma$ on $U$ is the graph
 with vertex set $U$ 
 and edge set consisting of all edges of $\Gamma$ that connect vertices in $U$. 
 The \textbf{join of two graphs} $\Gamma_1$ and $\Gamma_2$, denoted by $\Gamma_1\lor \Gamma_2$,  has a vertex set $V(\Gamma_1 \lor \Gamma_2) = V(\Gamma_1) \cup V(\Gamma_2)$ and an edge set $E(\Gamma_1 \lor \Gamma_2) = E(\Gamma_1) \cup E(\Gamma_2) \cup    \{x_1x_2: x_1 \in V(\Gamma_1), x_2 \in V(\Gamma_2)\}$.
A \textbf{complete
 graph}, $K_n$ is a graph with $n$ vertices which all pairs of vertices are joined.  A \textbf{clique} is a set of vertices whose induced subgraph is a complete graph. The maximum size
 of a clique in a graph $\Gamma$ is called the \textbf{clique number} of $\Gamma$.
     A vertex is considered a \textbf{dominating vertex} if it is adjacent to every other vertex in the graph.
  A subset $S\subseteq V(\Gamma)$ is called a
\textbf{dominating set} if every vertex of $V(\Gamma)\setminus S$ is adjacent to some vertex of $S$.  The minimum size of dominating sets of $\Gamma$, denoted by $\gamma(G)$, is called the \textbf{domination number} of $\Gamma$.  A \textbf{star graph} is a graph in which there is a vertex adjacent to all other vertices, with
 no further edges.

Throughout this paper, all groups considered are finite. $|G|$ denotes the cardinality of
 the set $G$ and $e$ denote its identity element of a group. For any two positive integers $m$ and $n$, the greatest common divisor, denoted by $gcd(m, n)$, represents the largest integer that divides both $m$ and $n$ without a remainder. In this case, $m$ and $n$ are also referred to as coprime numbers.
The cyclic group of order $n$ is denoted by $\mathbb{Z}_n$.
The subgroup generated by two vertices $x$ and $y$ is denoted by $\langle x, y \rangle$.
A \textbf{maximal subgroup} of a group is a proper subgroup that is not strictly contained in any other proper subgroup of the group.
A \textbf{maximal cyclic subgroup} of $G$ is a cyclic subgroup that is not properly contained within any other cyclic subgroup of $G$.
 A group $G$ is called a \textbf{$p$-group} for a prime $p$ if its order is given by $|G| = p^r$ for some $r \in \mathbb{N}$.  Thus a $p$-group is a group in which the order of every element is a power of $p$. A \textbf{Sylow $p$-subgroup} of a finite group $G$ is a maximal $p$-subgroup of $G$.

   $G$ is called a \textbf{nilpotent group} if its upper central series eventually terminates at $G$. Alternatively, its central series has finite length, or its lower central series ends at
 $\{e\}$. 

  A well-known characterization of nilpotent groups states that for a finite group $G$, the following conditions are equivalent:
     \begin{itemize}
         \item[$\bullet$] $G$ is nilpotent;
           \item[$\bullet$] $G$ decomposes as the direct product of its Sylow subgroups;
           \item[$\bullet$] For all prime divisors $p$ of $|G|$, the Sylow $p$-subgroup of $G$ is unique.
 \end{itemize}
 It is a well-established fact that every finite $p$-group and all abelian groups are nilpotent.
The symmetric group of degree $n$, denoted by $S_n$, is the group of all permutations of a set of $n$ elements.  For \( n \geq 3 \), the dihedral group of order \( 2n \) is denoted by \( D_{2n} \). It is defined by the presentation:
$$ D_{2n} = \langle a, b \mid  a^n = b^2 = e, b^{-1}ab = a^{-1} \rangle$$

 For $n\geq 2$, the generalized quaternion group $Q_{4n}$ of order $4n$ is given by
 $$ Q_{4n} = \langle x,y\mid x^n =y^2, \ x^{2n} = e, \ y^{-1}xy =x^{-1}\rangle $$
For more information about generalized quaternion groups see \cite{robinson2012course}. 
Definitions of the graphs required in this paper:
 

 \textbf{Power graph} of a group $G$, denoted by $Pow(G)$, is the graph with vertex set $G$, and two distinct vertices  $x$
 and $y$,  are adjacent if  $y\in \langle x\rangle$ or $x\in \langle y\rangle$. 

 The \textbf{enhanced power graph} of a group $G$ is the graph $P_e(G)$, where the vertex set is $G$, and two distinct vertices $x$ and $y$ are adjacent if  $\langle x, y \rangle$ is cyclic.


\textbf{The cyclic graph} of a group $G$, denoted by $C(G)$, is defined as the graph whose vertex set is $G\setminus Cyc(G)$, two vertices $x$ and $y$ are adjacent if $\langle x,y\rangle$ is cyclic. In this definition, isolated vertices are excluded, and these excluded vertices form the set $Cyc(G)$, given by:
  $$ \text{Cyc}(G)=\{x\in G\mid\langle y,x\rangle \text{ is cyclic for all } y\in G\}$$
 This graph is exactly the same as the enhanced power graph when dominating vertices are deleted. On the other hand, the cyclic graph is the induced subgraph of the enhanced power graph on the set $G\setminus Cyc(G)$, where $Cyc(G)$ is the set of dominating vertices of the enhanced power graph. 
\subsection{Motivation and outline of the paper}
This paper investigates the following question:

\textbf{Question:}  What structural similarities exist between the 
 the groups $G$ and $H$ if $P_e(G)\cong P_e(H)$?

We begin by reviewing relevant results in the literature on isomorphic graphs constructed on the elements of a group.
 In \cite{mirzargar2022finite}, we studied a similar question for power graphs, obtaining results on the possible orders of groups with isomorphic power graphs. Additionally, in \cite{abdollahi2007noncyclic}, the same question was explored for non-cyclic graphs. These studies motivated us to extend our investigation to enhanced power graphs. As demonstrated in \cite{cameron2011power, mirzargar2012power}, any two finite abelian groups with isomorphic power graphs must also be isomorphic to each other.
 Furthermore, Corollary 3.1 in \cite{zahirovic2020study} establishes that two groups have isomorphic power graphs if and only if their enhanced power graphs are isomorphic. This result implies that abelian groups with isomorphic enhanced power graphs must also be isomorphic. Additionally, in \cite{cameron2011power} it has been obtained that if $Pow(G)\cong Pow(H)$ then $G$ and $H$ have the same elements of each order.  This result also holds for enhanced power graphs.

In \cite{mirzargar2012power}, it has been established that certain well-known finite groups can be uniquely determined by their power graphs.
 Specifically, it is shown that if $G$ is a generalized quaternion group, a symmetric group, a cyclic group or a dihedral group, and $H$ is a finite group satisfying $Pow(G)\cong Pow(H)$, then $G\cong H$. Since an isomorphism between power graphs is equivalent to an isomorphism between enhanced power graphs, the same conclusion holds for enhanced power graphs. Note that the enhanced power graph of a group is a complete graph if and only if the group is cyclic. Furthermore, the uniqueness of the enhanced power graph for cyclic groups is evident. Thus, the enhanced power graph uniquely determines these groups up to isomorphism:
\begin{proposition}
     Let $G$ be a finite group. Then
    \begin{itemize}
           \item[i.] $P_e(G)\cong P_e(S_n)$ if and only if $G\cong S_n$, where $S_n$
 is the symmetric group of degree $n$.
         \item[ii.] $P_e(G)\cong P_e(D_{2n})$ if and only if $G\cong D_{2n}$, where $D_{2n}$
  is the dihedral group of order  $2n$.
        \item[iii.] $P_e(G)\cong P_e(Q_{4n})$ if and only if $G\cong Q_{4n}$, where $Q_{4n}$ is the generalized quaternion group of order $4n$.
 \end{itemize}
 \end{proposition}
 In \cite{ma2013cyclic}, Theorem 32, the same result is proven only for dihedral group $D_{2n}$ under the name of the "cyclic graph", where the cyclic graph in \cite{ma2013cyclic} is equivalent to the enhanced power graph. This paper is structured as follows: 
 
 In Section 2.1, we explore whether a group is uniquely determined by its enhanced power graph. Specifically, we investigate the conditions under which $P_e(G)\cong P_e(H)$ implies that $G\cong H $. In general, this does not always hold, so we focus on analyzing this question within the class of nilpotent groups.
Our main result establishes a structural property of nilpotent groups with respect to their enhanced power graphs.
\begin{theorem}\label{nilpotent}
     Let $G$ be a nilpotent group and $H$  an arbitrary group. Then $P_e(G)\cong P_e(H)$ if and only if $H$ is also nilpotent and $P_e(G_{P}) \cong P_e(H_{Q})$, for every prime $p$,  where $G_{P}$ and $H_{Q}$ represent the Sylow $p$-subgroups of $G$ and $H$, respectively.
\end{theorem} 
This result serves as the foundation for all subsequent findings, demonstrating that certain groups are uniquely determined by their enhanced power graphs.
 We further extend this analysis to the power graph and the cyclic graph.

In Section 2.2, we examine the nilpotent groups that share the same enhanced power graph. We characterize specific groups for which the enhanced power graphs are uniquely determined, which means that if for any group $H$, $P_e(G)\cong P_e(H)$  then $G\cong H$. Furthermore, we extend this uniqueness result to power graphs and cyclic graphs.
 Specifically, we establish the uniqueness of the enhanced power graph, power graph, and cyclic graph for the nilpotent groups $Q_8\times \mathbb{Z}_{n}$ ($n$ is odd), $(\prod_{i=1}^m \mathbb{Z}_2)\times \mathbb{Z}_n$ ($n$ is odd) and $\mathbb{Z}_p\times \mathbb{Z}_p\times \mathbb{Z}_n$ ($gcd(n,p)=1$). 
 Additionally, we analyze the structure of a group $G$ satisfying
  $P_e(G)\cong 
 P_e(\mathbb{Z}_p\times \mathbb{Z}_p\times \mathbb{Z}_p\times \mathbb{Z}_n)$ for a prime number $p > 2$
 and an integer $n > 0$ such that $n$ and $p$ are coprime.

The main focus of this paper is the enhanced power graph. However, for each theorem established, analogous results are also derived for both the power graph and the cyclic graph. The computations presented here were carried out using a program written in Julia for graph processing.

\section{Main Results}
\subsection{The structure of finite nilpotent groups with the same enhanced power graph}
In this section, we examine whether a group is uniquely determined by its enhanced power graph. Specifically, we investigate whether the isomorphism $P_e(G)\cong P_e(H)$ necessarily implies that $G\cong H $?

In general, this does not hold. Consequently, we focus our analysis on nilpotent groups. To demonstrate this, we present the following counterexample:\\
Consider the nilpotent groups $\mathbb{Z}_3\times \mathbb{Z}_3\times \mathbb{Z}_3$ and $S$ with the presentation
$S=\langle x,y | x^3=y^3=[x,y]^3=1\rangle$ where $[x,y]$ represents  the commutator $xyx^{-1}y^{-1}$.    Although these groups are not isomorphic, their enhanced power graphs are isomorphic (see Figure 1).
                      
\begin{figure}[h]
\label{figz}
\begin{center}
    \begin{tikzpicture}
        \node[circle, draw, fill=black, inner sep=2pt] (C) at (0,0) {};

        \foreach \i in {1,...,13} {
            \node[circle, draw, fill=black, inner sep=2pt] (A\i) at (360*\i/13:2) {};
            \node[circle, draw, fill=black, inner sep=2pt] (B\i) at (360*\i/13 + 360/26:3) {};

            \draw (C) -- (A\i);
            \draw (A\i) -- (B\i);
            \draw (B\i) -- (C);
        }
    \end{tikzpicture}
\end{center}

\caption{$P_{e}(\mathbb{Z}_3\times \mathbb{Z}_3\times \mathbb{Z}_3)\cong P_{e}(S)$ }
\end{figure}
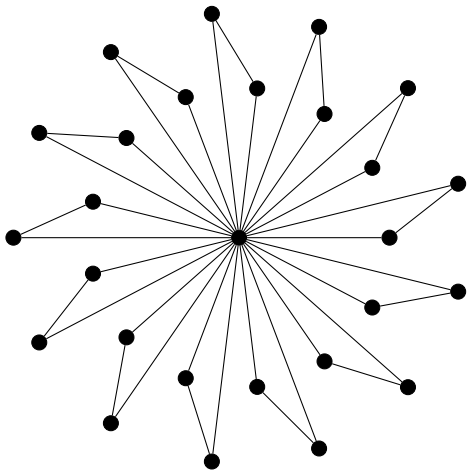

While an isomorphism between the enhanced power graphs of nilpotent groups does not necessarily imply an isomorphism between the groups themselves, the following lemma states the conditions under which such isomorphisms occur. Subsequently, in Theorem \ref{nilpotent}, we characterize the structure of finite nilpotent groups that possess the same enhanced power graph.

\begin{lemma} \label{lemma1}
Let $G$ be a nilpotent group, and $H$ be an arbitrary group. If the enhanced power graphs of $G$ and $H$ are isomorphic, i.e., $P_e(G)\cong P_e(H)$ then $H$ is necessarily nilpotent. 
\end{lemma}

\begin{proof}
    Assume that $P_e(G)\cong P_e(H)$. Then $G$ and $H$ have the same elements of each order and $|G|=|H|$. For each prime $p$ that divides the common order of the groups $G$ and $H$, the group $H$ has the same number of elements of order a power of $p$ as $G$. As $G$ is nilpotent, the count of elements whose order is a power of $p$ in $G$ matches the order of its Sylow $p$-subgroup. Consequently, the same holds for $H$, meaning that $H$ possesses a unique Sylow $p$-subgroup for every prime $p$ that divides its order. Therefore, $H$ is nilpotent.
    \end{proof}

\textit{ Proof of Theorem \ref{nilpotent}.}
    Let $G$ be a nilpotent group, $H$ an arbitrary group and suppose that $P_e(G)\cong P_e(H)$. Then  $G$ and $H$ must have the same order, and by Lemma \ref{lemma1}, both are nilpotent groups. Consequently, we have
 \begin{equation*}
     G = G_{P_1}\times G_{P_2}\times \cdots \times G_{P_r} ,  \ \ \  H = H_{Q_1}\times H_{Q_2}\times \cdots\times H_{Q_r}
 \end{equation*}
which $G_{P_i}$ and $H_{Q_i}$ are the unique Sylow $p_i$-subgroup of $G$ and $H$ respectively. Since $P_e(G)\cong P_e(H)$, then there is an isomorphism $\phi$ from $P_e(G)$ to $P_e(H)$.\\

 Let  $x= (x_1, x_2,\ldots,x_r)\in G$ and $y=(y_1, y_2,\ldots, y_r)\in  H$.  For each $1\leq i\leq r$  define $\psi_i :V(P_e(G_{P_i}))\rightarrow V(P_e(H_{Q_i}))$ by $x_i\rightarrow y_i$. We aim to show that $\psi_i$ is a graph isomorphism between $P_e(G_{P_i})$ and $P_e(H_{Q_i})$. 
 
 Since $\phi$ is a graph isomorphism and $\psi_i$ is defined in accordance with it,  it follows that $\psi_i$ must be a bijection. 
 Moreover, if $x_i\sim x’_i$ in $P_e(G_{P_i})$ then $\langle x_i,x’_i\rangle$ is cyclic subgroup of $G_{P_i}$. Consequently, the elements  $(e,…, e, x_i, e,…, e)$ and $(e,…, e, x’_i, e,…, e)$ form a cyclic subgroup of $G$. By the isomorphism $\phi$, we have
  $\phi((e,…, e, x_i, e,…, e)) \sim \phi((e,…, e, x’_i, e,…, e))$ in $P_e(H)$.  Therefore, the 
$i$-th coordinates of these images remain adjacent, which establishes that  $\psi_i$
  preserves adjacency. Hence, $\psi_i$ is a graph isomorphism.\\
  
Conversely, let $G$ and $H$ be nilpotent groups,
each containing a unique Sylow 
 $p_i$-subgroup for every prime $p_i$
  dividing their order.  For each prime $p_i$ let $G_{P_i}$ and $H_{Q_i}$ be the Sylow $p_i$-subgroups of $G$ and $H$, respectively. Suppose that for every  $1\leq i \leq r$, $P_e(G_{P_i}) \cong P_e(H_{Q_i})$.
Since there exists a graph isomorphism $\psi_i$ between  $P_e(G_{P_i})$ and $P_e(H_{Q_i})$, we define $\phi$ from $P_e(G)$ to $P_e(H)$ by  $$\phi((x_1,\ldots,x_i,\ldots,x_r))=( \psi_1(x_1),\ldots, \psi _i(x_i),\ldots, \psi_r(x_r)).$$
As $\phi$ is constructed based on the isomorphisms $\psi_i$,
it follows that $\phi$ is bijective. Now, suppose $(x_1,\ldots,x_r)\sim (x’_1,\ldots,x’_r)$ in $P_e(G)$. This implies that there exists $(z_1,\ldots,z_r)\in G$ such that $\langle(x_1,\ldots,x_r),(x’_1,\ldots,x’_r)\rangle \in \langle (z_1,\ldots,z_r)\rangle$. Consequently, for every $i$, $\langle x_i,x’_i\rangle \in \langle z_i\rangle$ which implies $x_i\sim x’_i$ in $P_e(G_{P_i})$. By graph isomorphism $\psi_i$, it follows that for each every $i$, $\psi_i( x_i)\sim \psi_i (x’_i)$ in $P_e(H_{Q_i})$. Thus, we conclude that $$( \psi_1(x_1),\ldots, \psi_i(x_i),\ldots, \psi_r(x_r)) \sim ( \psi_1(x’_1),\ldots, \psi_i(x’_i),\ldots, \psi_r(x’_r))$$ which implies $$\phi((x_1,\ldots,x_i,\ldots,x_r)) \sim \phi ((x’_1,\ldots,x’_i,\ldots,x’_r))$$
This establishes that $\phi$ is a graph isomorphism between $P_e(G)$ and $P_e(H)$, completing the proof.\qed


\vspace{0.5cm}
Building on our investigation of nilpotent groups with isomorphic enhanced power graphs,  we now focus on nilpotent groups with isomorphic power graphs.
To begin, we present a Corollary from \cite{zahirovic2020study}, which helps extend our results from enhanced power graphs to power graphs.

\begin{corollary}\label{cor}(Corollary 3.1, \cite{zahirovic2020study})
    For all finite groups $G$ and $H$ the following conditions are equivalent: 
\begin{enumerate}
    \item[i.] Undirected power graphs of $G$ and $H$ are isomorphic;
 \item[ii.]  Directed power graphs of $G$ and $H$ are isomorphic;
 \item[iii.] Enhanced power graphs of $G$ and $H$ are isomorphic.
\end{enumerate}
\end{corollary}

Additionally, we examine the cyclic graph of a group $G$, which consists of the vertex set $G\setminus Cyc(G)$ and two vertices $x$, $y$ are adjacent if $\langle x,y\rangle$ is cyclic.
In fact, the enhanced power graph satisfies the relation $P_e(G)=C(G)\vee Cyc(G) $. We also aim to extend our results on nilpotent groups with isomorphic enhanced power graphs to nilpotent groups with isomorphic cyclic graphs. Therefore, we explore the relationship between the cyclic graph and the enhanced power graph.
\begin{lemma}\label{lem2}
    For a pair of finite groups $G$ and $H$, $P_e(G)\cong P_e(H)$ if and only if $C(G)\cong C(H)$ and $|G|=|H|$. 
\end{lemma}

\begin{proof}
Suppose  $P_e(G)$ and $P_e(H)$ are isomorphic. Since the vertex sets of $P_e(G)$ and $P_e(H)$ are the elements groups $G$ and $H$, respectively, it follows that \(|G| = |H|\).  
Under the isomorphism, the  dominating set of $P_e(G)$, corresponding to \( \mathrm{Cyc}(G) \), must be mapped to the dominating set of $P_e(H)$, corresponding to \( \mathrm{Cyc}(H) \). Thus, we have \(|\mathrm{Cyc}(G)| = |\mathrm{Cyc}(H)|\).  
In other words, we also have
$${P_e(G)} = C(G) \vee \mathrm{Cyc}(G), \ \ \ P_e(H) = \ C(H) \vee \mathrm{Cyc}(H)$$.

Since $ P_e(G) \cong P_e(H)$ and $|\mathrm{Cyc}(G)|= |\mathrm{Cyc}(H)|$, it follows that  
$C(G) \cong C(H)$.

Conversely, suppose \( C(G) \cong C(H) \) and \(|G| = |H|\). By definition of the cyclic graphs, the vertex sets   \( G \setminus \mathrm{Cyc}(G) \) and \( H \setminus \mathrm{Cyc}(H) \) must have the same size. This implies that $|Cyc(G)|=|Cyc(H)|$ and
 \( P_e(G) \cong P_e(H) \).
\end{proof}

As shown in the proof of the Lemma \ref{lem2}, the condition \( |G| = |H| \) can be replaced with \( |\mathrm{Cyc}(G)| = |\mathrm{Cyc}(H)| \). This leads to the following corollary.

\begin{corollary}\label{condition} 
$P_e(G)\cong P_e(H)$ if and only if $C(G)\cong C(H)$ and $|\mathrm{Cyc}(G)| = |\mathrm{Cyc}(H)|$. 
\end{corollary}

\begin{theorem}\label{p}
     Let $G$ be a nilpotent group and $H$  an arbitrary group. Then $Pow(G)\cong Pow(H)$ if and only if $H$ is also nilpotent and $Pow(G_P)\cong Pow(H_Q)$  where $G_{P}$ and $H_{Q}$ represent the Sylow $p$-subgroups of $G$ and $H$, for every prime $p$, respectively.
\end{theorem} 
\begin{proof}
   Let $G$ be a nilpotent group.  If  $Pow(G)\cong Pow(H)$ then by the Corollary \ref{cor}, we get $P_e(G)\cong P_e(H)$. Additionally, Theorem \ref{nilpotent} asserts that $H$ is nilpotent as well, and for each prime $p$, the condition $P_e(G_{P}) \cong P_e(H_{Q})$ holds, where $G_{P}$ and $H_{Q}$ denote the Sylow $p$-subgroups of $G$ and $H$, respectively. Now, applying the Corollary \ref{cor}, to each   Sylow $p$-subgroups of $G$ and $H$, we conclude that  $Pow(G_{P}) \cong Pow(H_{Q})$. The converse can be proved by a similar way.
\end{proof}

We previously proved the following corollary in \cite{mirzargar2022finite}. Now, using Corollary \ref{cor} and Lemma \ref{lemma1} together, we derive the same conclusion. 
\begin{corollary}
   Let $G$ be a nilpotent group. If $H$ is any group for which $Pow(G)\cong Pow(H)$, then $H$ must also be nilpotent.
\end{corollary}

In \cite{abdollahi2007noncyclic}, Abdollahi established the following theorem for non-cyclic graph of nilpotent groups:
\begin{theorem}\label{abd}(\cite{abdollahi2007noncyclic})
    Let $G$ be a finite, non-cyclic nilpotent group. If $H$ is a group such that the number of cyclic elements in both $G$ and $H$ is equal to 1, and the non-cyclic graphs of $G$ and $H$ are isomorphic, then $H$ must also be a finite nilpotent group.
    Moreover, if $G_{P}$ and $H_{Q}$ are the Sylow $p$-subgroups of $G$ and $H$, for every prime $p$, respectively, the non-cyclic graph of $G_{P}$ is isomorphic to the non-cyclic graph of $H_{Q}$.
\end{theorem}
Using Lemma \ref{lem2} and Theorem \ref{nilpotent}, we eliminate some the additional conditions from the Theorem \ref{abd},  and establish the theorem  in both directions as follows:
\begin{theorem}\label{thm2}
      Let $G$ be a nilpotent group, $H$  an arbitrary group and $|G|=|H|$. Then $C(G) \cong C(H)$ if and only if $H$ is also nilpotent and $C(G_{P}) \cong C(H_{Q})$, for every prime $p$,  where $G_{P}$ and $H_{Q}$ represent the Sylow $p$-subgroups of $G$ and $H$, respectively.
\end{theorem} 
\begin{proof}
    Suppose $G$ be a nilpotent group and $C(G) \cong C(H)$ with $|G|=|H|$. By Lemma \ref{lem2}, we have $P_e(G)\cong P_e(H)$. Then by Theorem \ref{nilpotent}, $H$ is nilpotent and  for  Sylow $p$-subgroups $G_{P}$ of $G$ and $H_{Q}$ of $H$, $P_e(G_{P})\cong P_e(H_{Q})$. By removing the dominating vertices from  $P_e(G_{P})$ and $P_e(H_{Q})$  we conclude that   $C(G_{P}) \cong C(H_{Q})$.

    Conversely, suppose that $G$ and $H$ be nilpotent groups, $|G|=|H|$ and for every prime $p$, the Sylow $p$-subgroups $G_{P}$  of $G$ and $H_{Q}$  of $H$ satisfy $C(G_{P}) \cong C(H_{Q})$. It follows that  $|P|=|Q|$ since $|G|=|H|$ and $G_{P}$ and $H_{Q}$ are unique Sylow $p$-subgroups. Given that $C(G_{P}) \cong C(H_{Q})$ and $|G_P|=|H_Q|$ for every Sylow $p$-subgroup, we conclude that $P_e(G_P)\cong P_e(H_Q)$ for all Sylow $p$-subgroup of $G$ and $H$. Applying Theorem \ref{nilpotent}, we obtain $P_e(G)\cong P_e(H)$. Finally, by removing the dominating vertices from $P_e(G)$ and $P_e(H)$  we conclude that $C(G) \cong C(H)$.
\end{proof}
The following corollary arises directly from the proof of Theorem
 \ref{thm2}.
\begin{corollary}
    Let $G$ be a nilpotent group. If $H$ is any group for which $C(G)\cong C(H)$ and the orders of $G$ and $H$ are equal, then $H$ is also a nilpotent group.
\end{corollary}


\subsection{Nilpotent groups with the same  enhanced power graphs }
In this section, we examine nilpotent groups whose enhanced power graphs uniquely determine their group structure. Specifically, we identify certain groups \( G \) for which \( P_e(G) \cong P_e(H) \) for some group \( H \) implies \( G \cong H \). To achieve this, we utilize Theorem \ref{nilpotent} and analyze the Sylow subgroups of these groups. Additionally, we extend our investigation to power graphs and cyclic graphs, demonstrating analogous uniqueness results in these settings.
We begin by considering the graph $P_e(Q_8\times\mathbb{Z}_n)$,  (see Figure 2).
\begin{theorem}\label{q8zn}
Let $n$ be a positive odd integer, and let $G$ be a group such that $P_e(Q_8\times\mathbb{Z}_n)\cong P_e(G)$ if and only if $G \cong Q_8\times \mathbb{Z}_n$.
    \end{theorem}
\begin{proof}
Since $P_e(Q_8\times\mathbb{Z}_n)\cong P_e(G)$ and $Q_8\times \mathbb{Z}_n$ is a nilpotent group when $n$ is odd, from Lemma \ref{lemma1}, follows that $G$ is also nilpotent. 
We aim to find the Sylow subgroups of $G$ and express $G$ as a direct product of its Sylow subgroups. The group $Q_8\times \mathbb{Z}_n$ has a Sylow $2$-subgroup isomorphic $Q_8$ and for each prime $p$ that divides $n$, there exist a Sylow $p$-subgroup of order $p$-power isomorphic to $\mathbb{Z}_{p^k}$.  By Theorem \ref{nilpotent}, there are two cases:\\

\textit{Case 1.} $G$ has a Sylow $2$-subgroup $G_{P_1}$ of order 8 such that $P_e(Q_8) \cong P_e(G_{P_1})$. Thus $G_{P_1}$ is a group of order 8, $G_{P_1}$ and $Q_8$  have the same numbers of elements of each order. Since $Q_8$ is the unique group of order 8 with six elements of order 4, we conclude that $G_{P_1}\cong Q_8$. \\

\textit{Case 2.} $G$ has a Sylow $p_i$-subgroup $G_{P_i}$, for each prime $p_i$ dividing $n $, $2 \leq i\leq r$, such that $P_e(\mathbb{Z}_{p_i^k})\cong P_e(G_{P_i})$. Thus $P_e(G_{P_i})$ is a complete graph of order $G_{P_i^k}$, we conclude $G_{P_i}\cong \mathbb{Z}_{p_i^k}$.

  Therefore, when $n$ is odd,
 $$G \cong Q_8\times\prod_{p \mid n}\mathbb{Z}_{p^k}\cong Q_8\times \mathbb{Z}_{n}.$$
\end{proof}

\begin{figure}[h]
\begin{center}
    \begin{tikzpicture}
        \def\ovalW{1.5}  
        \def\ovalH{1}    
        
        \node[draw, ellipse, minimum width=\ovalW cm, minimum height=\ovalH cm] (C) at (0,0) {\Huge $K_{2n}$};
        
        \node[draw, ellipse, minimum width=\ovalW cm, minimum height=\ovalH cm] (A) at (-2,2) {\Huge $K_{2n}$};
        \node[draw, ellipse, minimum width=\ovalW cm, minimum height=\ovalH cm] (B) at (2,2) {\Huge $K_{2n}$};
        \node[draw, ellipse, minimum width=\ovalW cm, minimum height=\ovalH cm] (D) at (0,-2.5) {\Huge $K_{2n}$};
        
        \draw (C) -- (A);
        \draw (C) -- (B);
        \draw (C) -- (D);
    \end{tikzpicture}
\end{center}
\caption{The enhanced graph of $Q_8\times \mathbb{Z}_{n}$  where $n$ is an odd integer, consist of  three complete graphs $K_{2n}$ with $2n$ dominating vertices.}
\end{figure}
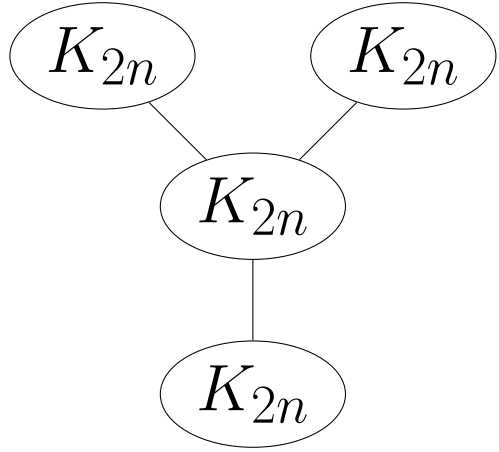

Now by Corollary \ref{cor}, we conclude that the same result holds for the power graphs. This implies that the power graphs of $Q_8\times \mathbb{Z}_{n}$, when $n$ is odd, are unique.

\begin{corollary}
Let $n$ be a positive odd integer.
  $G$ is a group such that $Pow(Q_8\times \mathbb{Z}_n) \cong Pow(G)$ if and only if $G \cong Q_8\times \mathbb{Z}_n$.
    \end{corollary}

Next, we analyze the graph $P_e((\prod_{i=1}^{m} 
    \mathbb{Z}_2)\times\mathbb{Z}_n)$, as illustrated in Figure 3.

\begin{theorem}
    Let $G$ be a group and $n$ be a positive odd integer. Then $P_e(G)\cong P_e((\prod_{i=1}^{m} 
    \mathbb{Z}_2)\times\mathbb{Z}_n)$ if and only if $G\cong (\prod_{i=1}^m \mathbb{Z}_2)\times \mathbb{Z}_n$. 
\end{theorem}
\begin{proof}
    Since $P_e((\prod_{i=1}^m \mathbb{Z}_2)\times \mathbb{Z}_n) \cong P_e(G)$ and $(\prod_{i=1}^m \mathbb{Z}_2)\times \mathbb{Z}_n$ is a nilpotent group when $n$ is odd, it follows that $G$ is also nilpotent with a Sylow $2$-subgroup isomorphic $\prod_{i=1}^m \mathbb{Z}_{2}$ and for every prime $p$ that divides $n$, there exist  a Sylow 
    $p$-subgroup of order $p$-power isomorphic to $\mathbb{Z}_{p^k}$. Similarly, by Theorem \ref{nilpotent}, $G$ has the following Sylow subgroups:
\begin{itemize}
    \item [i.] A Sylow $2$-subgroup $G_{P}$ such that $P_e((\prod_{i=1}^m \mathbb{Z}_{2})) \cong P_e(G_P)$. It is easy to see that the enhanced power graph of a group is a star graph if and only if it is a finite group such that  $o(x) = 2$ for every non-identity element $x$. This implies the group is abelian and isomorphic to $ \mathbb{Z}_2\times \mathbb{Z}_2 \cdots \times \mathbb{Z}_2$. Thus, $G_P\cong \prod_{i=1}^m \mathbb{Z}_2$.

    \item[ii.] For each prime $p_i$ dividing $n $, a Sylow $p_i$-subgroup $G_{P_i}$, such that $P_e(\mathbb{Z}_{p_i^k})\cong P_e(G_{P_i})$. This implies $G_{P_i}\cong \mathbb{Z}_{p_i^k}$. 
    \end{itemize}
Therefore, when $n$ be a positive odd integer, 
$ G\cong (\prod_{i=1}^{m} \mathbb{Z}_2)\times\mathbb{Z}_{p_i^k} \cong (\prod_{i=1}^{m} \mathbb{Z}_2)\times\mathbb{Z}_n  $. 
\end{proof}

Now by Corollary \ref{cor}, we conclude that the same result holds for the power graphs.
\begin{corollary}
  Let $G$ be a group and  $n$ be a positive odd integer. Then $Pow(G)\cong Pow((\prod_{i=1}^{m} \mathbb{Z}_2)\times \mathbb{Z}_n)$ if and only if $G\cong (\prod_{i=1}^m \mathbb{Z}_2)\times \mathbb{Z}_n$.    
  \end{corollary}

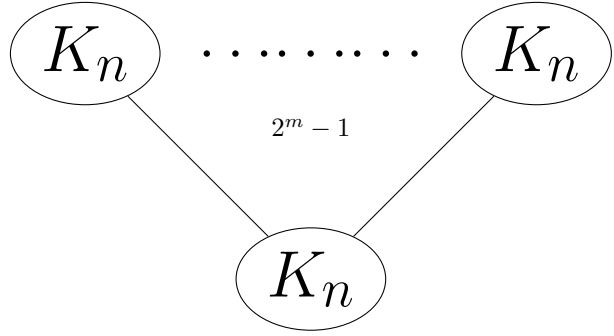
\begin{figure}[h]
\begin{center}
    \begin{tikzpicture}
        \def\ovalW{1.5}  
        \def\ovalH{1}    
        
        \node[draw, ellipse, minimum width=\ovalW cm, minimum height=\ovalH cm] (C) at (0,0) {\Huge $K_{n}$};
        
        \node[draw, ellipse, minimum width=\ovalW cm, minimum height=\ovalH cm] (A) at (3,3) {\Huge $K_{n}$};

        \node (S) at (-1,3) {\Huge $\ldots$};
        \node (S) at (0,3) {\Huge $\ldots$};
        \node (S) at (0,2) { $2^m-1$};
        \node (S) at (1,3) {\Huge $\ldots$};
        \node[draw, ellipse, minimum width=\ovalW cm, minimum height=\ovalH cm] (B) at (-3,3) {\Huge $K_{n}$};

        \draw (C) -- (A);
        \draw (C) -- (B);
        
    \end{tikzpicture}
\end{center}
\caption{The enchanced graph of $ (\prod_{i=1}^m \mathbb{Z}_2)\times \mathbb{Z}_n$, where $n$ is an odd integer, consist of  $2^m-1$ complete graphs $K_{n}$ with $n$ dominating vertices.}
\end{figure}

Next, we examine the graph $P_e(\mathbb{Z}_p\times\mathbb{Z}_p\times \mathbb{Z}_n)$, $n$ is odd , as depicted in Figure 4.

\begin{theorem}
     Let $G$ be a group, $p$ a prime, and $n > 0$ an integer such that $n$ and $p$ are coprime. Then $P_e(G)\cong P_e(\mathbb{Z}_p\times\mathbb{Z}_p\times \mathbb{Z}_n)$ if and only if $G\cong \mathbb{Z}_p\times\mathbb{Z}_p\times \mathbb{Z}_n$. 
\end{theorem}
\begin{proof}
 Since $\mathbb{Z}_p\times\mathbb{Z}_p\times\mathbb{Z}_n$ is nilpotent, by Lemma \ref{lemma1}, $G$ is also nilpotent. Now by Theorem \ref{nilpotent}, we search for Sylow subgroups $\mathbb{Z}_p\times\mathbb{Z}_p\times\mathbb{Z}_n$.  Its Sylow subgroups are the following:
 \begin{enumerate}
     \item[i.] A Sylow subgroup of order $p^2$ isomorphic to $\mathbb{Z}_p\times\mathbb{Z}_p$. Thus by Theorem \ref{nilpotent}, $G$ has a $p$-Sylow subgroup $G_{P}$ such that $P_e(G_P)\cong P_e(\mathbb{Z}_p\times\mathbb{Z}_p)$. Since there are exactly two non-isomorphic groups of order $p^2$, $\mathbb{Z}_p\times\mathbb{Z}_p$ and $\mathbb{Z}_{p^2}$ and $P_e(\mathbb{Z}_{p^2})$ is a complete graph, we conclude that $G_P\cong \mathbb{Z}_p\times \mathbb{Z}_p$.
     \item[ii.] For each prime $q$ dividing $n$, a Sylow $q$-subgroup of $\mathbb{Z}_p\times \mathbb{Z}_p\times \mathbb{Z}_n$ isomorphic to $\mathbb{Z}_{q^k}$. Thus $G$ has a Sylow $q$-subgroup $G_P$, for each prime $q$ dividing $n$, such that $P_e(G_P)\cong P_e(\mathbb{Z}_{q^k})$. Since, $P_e(\mathbb{Z}_{q^k})$ is a complete graph  $G_P \cong \mathbb{Z}_{q^k} $.
 \end{enumerate}
 Therefore, we express $G$ as a direct product of its Sylow subgroups. For each prime $q$ dividing $n$, $G\cong \mathbb{Z}_p\times \mathbb{Z}_p\times \mathbb{Z}_{q^k}\cong  \mathbb{Z}_p\times \mathbb{Z}_p\times \mathbb{Z}_n$.
\end{proof}

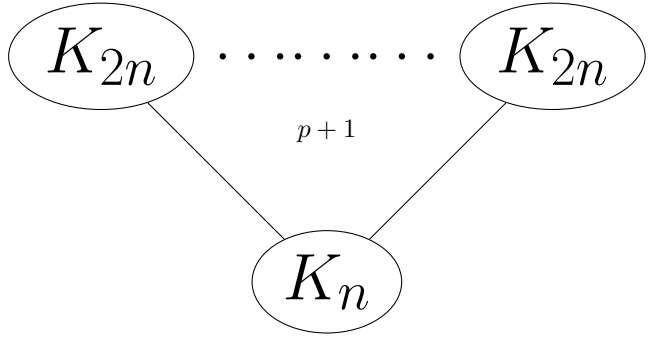
\begin{figure}[h]
\begin{center}
    \begin{tikzpicture}
        \def\ovalW{1.5}  
        \def\ovalH{1}    
        
        \node[draw, ellipse, minimum width=\ovalW cm, minimum height=\ovalH cm] (C) at (0,0) {\Huge $K_{n}$};
        
        \node[draw, ellipse, minimum width=\ovalW cm, minimum height=\ovalH cm] (A) at (3,3) {\Huge $K_{2n}$};

        \node (S) at (-1,3) {\Huge $\ldots$};
        \node (S) at (0,3) {\Huge $\ldots$};
        \node (S) at (0,2) { $p+1$};
        \node (S) at (1,3) {\Huge $\ldots$};
        \node[draw, ellipse, minimum width=\ovalW cm, minimum height=\ovalH cm] (B) at (-3,3) {\Huge $K_{2n}$};

        \draw (C) -- (A);
        \draw (C) -- (B);

    \end{tikzpicture}
\end{center}
\caption{The enchanced graph of $ \mathbb{Z}_p\times \mathbb{Z}_p\times \mathbb{Z}_{n}$, where $gcd(n,p)=1$, consist of $p+1$ complete graph $K_{2n}$ with $n$ dominating vertices.}
\end{figure}

Now by Lemma \ref{lem2}, we conclude that the same result holds for the power graphs.
\begin{corollary}
     Let $G$ be a group, $p$ a prime, and $n > 0$ an integer such that $n$ and $p$ are coprime. Then $Pow(G)\cong Pow(\mathbb{Z}_p\times \mathbb{Z}_p\times \mathbb{Z}_n)$ if and only if $G\cong \mathbb{Z}_p\times \mathbb{Z}_p\times \mathbb{Z}_n$. 
\end{corollary}

\begin{theorem}
     Let $G$ be any group. Then  $P_e(G)\cong 
 P_e(\mathbb{Z}_p\times \mathbb{Z}_p\times \mathbb{Z}_p\times \mathbb{Z}_n)$ for a prime number $p > 2$
 and an integer $n > 0$ such that $n$ and $p$ are coprime iff 
 $G\cong\mathbb{Z}_p\times\mathbb{Z}_p \times\mathbb{Z}_p\times \mathbb{Z}_n$ or $K\times\mathbb{Z}_n$, where $K$ represents the unique non-abelian group of order $p^3$ and exponent $p$.
.
\end{theorem}
\begin{proof}
    Note that $\mathbb{Z}_p\times\mathbb{Z}_p\times \mathbb{Z}_p \times \mathbb{Z}_n$ is nilpotent when $gcd(n,p) = 1$. By Corollary \ref{lemma1}, $G$ is nilpotent. Now  by theorem \ref{nilpotent}, consider the Sylow subgroups of $\mathbb{Z}_p \times \mathbb{Z}_p\times\mathbb{Z}_p\times \mathbb{Z}_n$ and $G$.
    \begin{enumerate}
        \item  Since $\mathbb{Z}_p\times \mathbb{Z}_p\times \mathbb{Z}_p\times\mathbb{Z}_n$ has a Sylow $p$-subgroup isomorphic to $\mathbb{Z}_p\times\mathbb{Z}_p\times \mathbb{Z}_p$, $G$ also has a Sylow $p$-subgroup $G_{P}$  which $P_e(\mathbb{Z}_p \times\mathbb{Z}_p \times \mathbb{Z}_p)\cong P_e(G_P)$. Thus $\mathbb{Z}_p\times \mathbb{Z}_p\times\mathbb{Z}_p $ and $G_P$ have the same numbers of elements of each order, identity of order 1 and $p^3-1$ elements of order $p$. 
        Among the five possible groups of order $p^3$, two exhibit the same element order distribution as  $\mathbb{Z}_p \times\mathbb{Z}_p\times \mathbb{Z}_p $:  the abelian group $\mathbb{Z}_p\times\mathbb{Z}_p\times\mathbb{Z}_p $ itself and the unique non-abelian group of order  $p^3$ with exponent $p$, denoted by $K$, which is defined by the presentation 
        $$\langle x, y, z \ |\ x^p=y^p=z^p=1, [x,y]=z, [x,z]= [y,z]=1\rangle$$
 Therefore from $P_e(\mathbb{Z}_p\times\mathbb{Z}_p\times \mathbb{Z}_p)\cong P_e(G_P)$ we conclude that $G_P\cong \mathbb{Z}_p \times \mathbb{Z}_p \times \mathbb{Z}_p $ or $G_P \cong K$
        \item  For each prime $q$ dividing $n$,  $\mathbb{Z}_p \times \mathbb{Z}_p \times\mathbb{Z}_p\times \mathbb{Z}_n$ has a Sylow $q$-subgroup isomorphic to $\mathbb{Z}_{q^k}$. Thus, $G$ has a  Sylow $q$-subgroup $G_P$ such that $P_e(G_P)\cong P_e(\mathbb{Z}_{q^k})$ which implies $G_P\cong Z_{q^k} $, for each  prime $q$ dividing $n$. 
    \end{enumerate}
    Now, we express $G$ as a direct product of its Sylow subgroups. For each  prime $q$ dividing $n$,
      $G\cong \mathbb{Z}_p\times\mathbb{Z}_p\times \mathbb{Z}_p\times\mathbb{Z}_{q^k}\cong \mathbb{Z}_p\times \mathbb{Z}_p \times \mathbb{Z}_p\times\mathbb{Z}_n$ or $G\cong K\times \mathbb{Z}_{q^k}\cong K\times \mathbb{Z}_n$. 
  \end{proof}
  Now by Corollary \ref{cor}, we conclude that the same result holds for the power graphs.
\begin{corollary}
     Let $G$ be a group. $Pow(G)\cong 
 Pow(\mathbb{Z}_p\times\mathbb{Z}_p\times\mathbb{Z}_p\times \mathbb{Z}_n)$ for a prime number $p > 2$
 and an integer $n > 0$ such that $gcd(n,p) = 1$ if and only if $G\cong \mathbb{Z}_p\times\mathbb{Z}_p\times\mathbb{Z}_p \times \mathbb{Z}_n$ or $K\times\mathbb{Z}_n$, where $K$ is unique non-abelian group of order $p^{3}$ and exponent $p$.
\end{corollary}

\textbf{Remark.} Similar to the section 2.1, we aim to investigate the $C(G)\cong C(H)$ when $H$ is is one of the following groups: $Q_8\times \mathbb{Z}_n$ ($n$ is odd), $(\prod_{i=1}^m \mathbb{Z}_2)\times\mathbb{Z}_n$ ($n$ is odd), $\mathbb{Z}_p\times \mathbb{Z}_p\times\mathbb{Z}_n$ ($gcd(n, p) = 1$).
In \cite{abdollahi2007noncyclic}, Abdollahi et al. demonstrated that if $G$ is a group whose non-cyclic graph is isomorphic to the non-cyclic graph of one of these groups, then it follows that $G \cong H$.
 
Since graph isomorphism also extends to the complement of the non-cyclic graph, it follows that  $C(G)\cong C(H)$ if and only if $G\cong H$, where $H$ belongs to the aforementioned family of groups. 
Furthermore, in \cite{abdollahi2007noncyclic}, it was also proven that for a prime $p > 2$ and an positive integer $n$ such that $gcd(n,p) = 1$, the non-cyclic graph $G$ isomorphic the non-cyclic graph $(\mathbb{Z}_p\times \mathbb{Z}_p\times \mathbb{Z}_p\times \mathbb{Z}_n)$ if and only if 
 $G\cong \mathbb{Z}_p \times\mathbb{Z}_p \times \mathbb{Z}_p \times\mathbb{Z}_n$ or $K\times \mathbb{Z}_n$, where $K$ is the unique non-abelian group with order $p^3$ and exponent $p$.
Consequently, the same conclusion holds for the cyclic graph.

\section*{Acknowledgment}
 We wish to express our sincere gratitude to the Visiting or Sabbatical Scientist Support Program for their invaluable assistance. Furthermore, the first author acknowledges the financial support provided by the Scientific and Technological Research Council of Turkey (TUBITAK) through the Bideb 2221 program.

\end{document}